\documentclass[12pt,leqno]{article}
\usepackage{amsmath,amssymb,amsbsy,amsfonts,amsthm,latexsym,
         amsopn,amstext,amsxtra,euscript,amscd,amsthm,mathdots}

\allowdisplaybreaks

\numberwithin{equation}{section}

\newtheorem{lem}{Lemma}

\newtheorem{thm}{Theorem}

\begin{document}

\begin{large}
\centerline{\Large \bf A hybrid inequality for the number}
\centerline{\Large \bf of divisors of an integer}

\end{large}
\vskip 10pt
\begin{large}
\centerline{\sc  Patrick Letendre}
\end{large}
\vskip 10pt
\begin{abstract}
We establish an explicit inequality for the number of divisors of an integer $n$. It uses the size of $n$ and its number of distinct prime divisors.
\end{abstract}


\section{Introduction and notation}

Let $\tau(n)$ and $\omega(n)$ be respectively the number of divisors and the number of distinct prime factors of $n$. In \cite{jmdk:pl}, the author and De Koninck have studied a variety of inequalities for the $\tau$ function. Among many helpful comments and suggestions, the anonymous referee of the said paper asked to justify and clarify some preliminary statements concerning the quality of our inequalities when compared to the well-known theorem of Wigert \cite{sw}
$$
\tau(n) \le 2^{\frac{\log n}{\log\log n}(1+o(1))} \quad n \rightarrow \infty.
$$
The present author has therefore reworked some of his results to get to a nice statement which is also inspired by Théorème 2 of \cite{pe:jln} and by \cite{jln:gr}. Let us define $\vartheta:=\vartheta(n)$ by $\omega(n) = \vartheta\frac{\log n}{\log\log n}$ for each $n \ge 16$. Then,
$$
\tau(n) \le \exp\biggl(\vartheta\log\biggl(1+\frac{1}{\vartheta}\biggr)\frac{\log n}{\log\log n}\biggl(1+O\biggl(\frac{\log\log\log n}{\log\log n}\biggr)\biggr)\biggr).
$$
In this paper, we make this result explicit and the best possible constant, here implicit, is found. A more precise result, not mentioned in \cite{jmdk:pl}, is also obtained in Theorem \ref{thm_2}.

Let us define the function
\begin{equation}\label{rho}
\rho(n) := \biggl(\frac{\log \tau(n)}{\omega(n)\log\bigl(1+\frac{1}{\vartheta(n)}\bigr)}-1\biggr)\frac{\log\log n}{\log\log\log n}.
\end{equation}

\begin{thm}\label{thm_1}
For each integer $n \ge 17$,
$$
\rho(n) \le \rho(2^{26}3^{16}) = 2.00080128 \dots
$$
The inequality is strict unless $n = 2^{26}3^{16}$.
\end{thm}

Let $\mathcal{K}$ be the convex hull of the set $\bigl\{\bigl(\frac{1}{s},\frac{\log(s+1)}{s}\bigr) :\ s\in\mathbb{N}\bigr\} \cup \{(0,0)\}$. We thus consider the function $f:[0,1] \rightarrow \mathbb{R}$ defined by
$$
f(\vartheta):=\max \{(\vartheta,z) :\ z \in \mathbb{R}\} \cap \mathcal{K}.
$$

\begin{thm}\label{thm_2}
Let $\vartheta \in (0,1]$ be fixed. Let $\mathcal{J}$ be the ordered sequence of integers $n$ satisfying
$$
\Bigl|\omega(n) - \frac{\vartheta \log n}{\log\log n}\Bigr| < \frac{\log n}{(\log\log n)^{3/2}}.
$$
We have
$$
\limsup_{\substack{n \rightarrow \infty \\ n \in \mathcal{J}}} \frac{\log \tau(n)\log\log n}{\log n} = f(\vartheta).
$$
\end{thm}

For each integer $k \ge 0$ we define $n_k$ by $n_k=p_1 \cdots p_k$ (so that $n_0=1$) where $p_k$ is the $k$-th prime number. We say that an integer is {\it primary} if it can be written as
$$
n=p_1^{\alpha_1}\cdots p_k^{\alpha_k} \qquad \alpha_1 \ge \cdots \ge \alpha_k
$$
for some $k \ge 0$.

\section{Preliminary lemmas}

A well-known consequence of the prime number theorem is that
$$
\max_{n \le z}\omega(n)=\frac{\log z}{\log\log z}+O\Bigl(\frac{\log z}{(\log\log z)^2}\Bigr).
$$
We need an explicit upper bound for our result.

\begin{lem}\label{lem_1}
For each integer $n \ge 3$ we have
$$
\omega(n) \le \kappa\frac{\log n}{\log\log n}
$$
where $\kappa=1.3840127\dots$ The inequality is strict unless $n=n_9$.
\end{lem}

\begin{proof}
See \cite{gr}. It is also a consequence of \eqref{lem2-5} below along with some computations.
\end{proof}

\begin{lem}
We have
\begin{eqnarray}\label{lem2-1}
\sum_{j=1}^{k} \log\log p_j & \ge & k\log\log k \quad k \ge 44,\\ \label{lem2-2}
\log n_k & \le & 2k\log k \quad k \ge 2,\\ \label{lem2-3}
\log\log n_k & \ge & \log k \quad k \ge 3,\\ \label{lem2-4}
\log n_k & \le & k\log\log n_k \quad k \ge 3.
\end{eqnarray}
\end{lem}

\begin{proof}
\eqref{lem2-1} and \eqref{lem2-2} are simple consequences of Lemma 4.8 of \cite{jmdk:pl} in (4.22) and (4.20) respectively. \eqref{lem2-3} is done by induction. For \eqref{lem2-4}, we first establish the inequality
\begin{equation}\label{lem2-5}
\log n_k  \ge  k(\log k+\log\log k-5/4) \quad k \ge 2
\end{equation}
with induction by using the fact that $p_k \ge k\log k$ from \cite{jbr}. So, we get a lower bound for $k\log\log n_k$ with \eqref{lem2-5} and an upper bound for $\log n_k$ with (4.20) from Lemma 4.8 of \cite{jmdk:pl}. We let the details to the reader.
\end{proof}

\begin{lem}\label{lem_3}
For each fixed $k \in \mathbb{N}$ and $c \in \mathbb{R}_{>0}$, the function
\begin{equation}\label{ins_f}
\log\Bigl(1+\frac{\exp(x)}{kx}\Bigr)\Bigl(1+\frac{c\log x}{x}\Bigr)
\end{equation}
is strictly increasing for $x \ge 1$.
\end{lem}

\begin{proof}
The derivative of \eqref{ins_f} with respect to $x$ is
$$
\frac{-c(\log x-1)(1+\vartheta)\log(1+\frac{1}{\vartheta})+(x-1)(c\log x+x)}{x^2(1+\vartheta)},
$$
where we write $\vartheta=\frac{kx}{\exp(x)}$ to simplify. We want to show that the numerator is positive for $x \ge 1$. We will show that the function
$$
-(\log x-1)(1+\vartheta)\log(1+\frac{1}{\vartheta})+(x-1)\log x
$$
is positive for $x \ge 1$. It is the derivative of the numerator with respect to $c$. We observe that it implies the desired result for each fixed $c > 0$ and $x \ge 1$. 

This result clearly holds when $x \in [1,\exp(1)]$. For $x > \exp(1)$, since the function $z \mapsto (1+z)\log\bigl(1+\frac{1}{z}\bigr)$ is strictly decreasing for $z > 0$, it is enough to prove the result with $k=1$. We then use the inequality $\log 1+\frac{\exp(x)}{x} \le \log \frac{\exp(x+1)}{x}$ to deduce that it is enough to prove that
$$
-(\log x-1)(1+\frac{x}{\exp(x)})(x+1-\log x)+(x-1)\log x \ge 0
$$
for $x > \exp(1)$. We observe that it can be written as
$$
(x+1-3\log x)+\Bigl(\log^2 x-\frac{(x+1)^2\log x}{\exp(x)}\Bigr)+\frac{x^2+x+x\log^2 x+\log x}{\exp(x)}
$$
and the result follows from the fact that each of the three terms is positive for $x > \exp(1)$.
\end{proof}

Let $\gamma(n)$ stand for the product of the distinct prime factors of $n$.

\begin{lem}\label{lem_4}
For each integer $n \ge 2$,
\begin{equation}\label{ram}
\tau(n) \le \prod_{p \mid n}\Bigl(\frac{\log n\gamma(n)}{\omega(n)\log p}\Bigr).
\end{equation}
\end{lem}

\begin{proof}
This is a famous inequality due to Ramanujan \cite{sr}. We can find a proof in Corollary 4.5 of \cite{jmdk:pl} as well.
\end{proof}

\begin{lem}\label{lem_5}
Let $\omega(n)=:k \ge 74$. Then,
\begin{equation}\label{res_1}
\tau(n) < \Bigl(1+\frac{\log n}{k\log k}\Bigr)^k.
\end{equation}
\end{lem}

\begin{proof}
This is Theorem 3.4 from \cite{jmdk:pl}. In this paper we are using this result only for $k \ge 95$, which requires substantially fewer computations.
\end{proof}

\begin{lem}\label{lem_6}
Let $\delta > 0$ and $[\alpha,\beta] \subseteq [a,b]$ be fixed. Let also $h \in C([a,b])$ satisfying $\max_{x \in [\alpha,\beta]} |h'(x)| \le M_1$. Assume that
$$
\max_{\substack{i \in \mathbb{Z}\\ \mu+i\delta \in [a,b]}}h(\mu+i\delta) \le M
$$
for some $\mu \in [\alpha,\beta]$. Then,
$$
\max_{x \in [\alpha,\beta]}h(x) < M+\delta M_1.
$$
\end{lem}

\begin{proof}
Let $z \in [\alpha,\beta]$ be fixed. There is a $j \in \mathbb{Z}$ with $|\mu+j\delta-z| < \delta$. The result follows from
$$
h(z)=h(\mu+j\delta)+\int_{\mu+j\delta}^{z}h'(t){\rm d}t.
$$
\end{proof}

\section{Proof of Theorem \ref{thm_1}}

Throughout this proof, we often write $x=\log\log n$ to simplify the notation. Also, $x$ is sometime considered as a real variable when arguments from calculus have to be used. It should not be a problem for the reader. Furthermore, it is always assumed that $k=\omega(n)$ is fixed. Lemma \ref{lem_3} allows us to assume that $n$ is primary when $\omega(n) = k \ge 3$. We have used PARI/GP to verify that $\rho(n) < 2$ for each $17 \le n \le 10^9$. In particular, this verification along with Lemma \ref{lem_3} leave us with only primary integers to verify in the case where $k=1,2$ too.

\subsection{The case $k \ge 11000$}\label{t1.1}

In this section, we will establish that $\rho(n) < 2$ for all the integers with at least $11000$ distinct prime factors. Our main tool is \eqref{res_1}. Let us write
$$
1+\frac{\exp(x)}{k\log k}=1+\frac{1}{\vartheta}+\Bigl(\frac{\exp(x)}{k\log k}-\frac{1}{\vartheta}\Bigr)
$$
where the term in parenthesis is positive if and only if $x \ge \log k$. If it is negative, then the result follows directly of \eqref{res_1}. In the case where it is positive, we use the mean value theorem to get to
\begin{eqnarray*}
\log\Bigl(1+\frac{\exp(x)}{k\log k}\Bigr) &=& \log\Bigl(1+\frac{1}{\vartheta}+\Bigl(\frac{\exp(x)}{k\log k}-\frac{1}{\vartheta}\Bigr)\Bigr)\\
&\le& \log\Bigl(1+\frac{1}{\vartheta}\Bigr)+\frac{1}{1+\frac{1}{\vartheta}}\Bigl(\frac{\exp(x)}{k\log k}-\frac{1}{\vartheta}\Bigr)\\
&=& \log\Bigl(1+\frac{1}{\vartheta}\Bigr)\Bigl(1+\frac{1}{(\vartheta+1)\log(1+\frac{1}{\vartheta})}\frac{x-\log k}{\log k}\Bigr)
\end{eqnarray*}
Now, since $(\vartheta+1)\log(1+\frac{1}{\vartheta}) \ge \log(\frac{1}{\vartheta})=:t$ and $t=x-\log x-\log k$, we deduce that $\rho(n) < 2$ holds if
\begin{eqnarray*}
\frac{\log x+t}{t(x-\log x-t)} & < & \frac{2\log x}{x}
\end{eqnarray*}
which is the case if $1 \le t \le \frac{x}{2}$ when $x \ge 11.66$. Indeed, by expanding, we find a parabola in $t$ so that it is enough to verify at $t=1$ and at $t=\frac{x}{2}$. From there, it is an easy exercise that uses calculus. The details are left to the reader.

In the case where $t > \frac{x}{2}$, we use the fact that $\tau(n) < (\frac{2\log n}{k\log k})^k < (\frac{\log n}{k})^k$ (since $2^k \le \tau(n) < (1+\frac{\log n}{k\log k})^k$) from Lemma \ref{lem_5}. Thus, since we have
\begin{eqnarray*}
\log\Bigl(1+\frac{1}{\vartheta}\Bigr)\Bigl(1+\frac{2\log x}{x}\Bigr) &>& \log\Bigl(\frac{1}{\vartheta}\Bigr)\Bigl(1+\frac{2\log x}{x}\Bigr)\\
&=& (x-\log x-\log k)\Bigl(1+\frac{2\log x}{x}\Bigr)\\
&=& (x-\log x-\log k)+t\frac{2\log x}{x}\\
&>& x-\log k,\\
\end{eqnarray*}
the result follows.

The remaining case is when $t < 1$, i.e. when $\frac{1}{\exp(1)} < \vartheta < 1.39$ (see Lemma \ref{lem_1}). We then have
$$
\frac{1}{(\vartheta+1)\log(1+\frac{1}{\vartheta})}\frac{x-\log k}{\log k} \le \frac{\log x+1}{1.25(x-\log x-1)} < \frac{2\log x}{x}
$$
for $x \ge 11.66$. Since $\log\log n_{11000} > 11.66$, the result follows.

\subsection{The case $44 \le k \le 10999$}

We use inequality \eqref{ram} into the definition of $\rho$ \eqref{rho} to get to
\begin{eqnarray}\label{s3,i1}
\rho(n) & \le & \Bigl(\frac{\log(\exp(x)+\log n_k)-\log k-\frac{1}{k}\sum_{j=1}^{k}\log\log p_j}{\log(1+\frac{1}{\vartheta})}-1\Bigr)\frac{x}{\log x}\\ \label{s3,i2}
& = & \frac{x\log(\frac{\exp(x)+\log n_k}{\exp(x)+kx})+x\log x-\frac{x}{k}\sum_{j=1}^{k}\log\log p_j}{\log x\log(1+\frac{\exp(x)}{kx})}.
\end{eqnarray}

Since $x \ge \log\log n_k$ and from the inequality $\log n_k \le k\log\log n_k$ \eqref{lem2-4}, we deduce that the first term in the numerator in \eqref{s3,i2} is negative. Thus, from $\sum_{j=1}^{k}\log\log p_j > k\log\log k$ \eqref{lem2-1} and $\log(1+\frac{\exp(x)}{kx}) > x-\log x-\log k$, we find that $\rho(n) < 2$ if
\begin{equation}\label{i:1.2.1}
x\log x+x\log\log k-2\log x(\log x+\log k) > 0
\end{equation}
which is the case when $x \ge 1.8\log k$. To prove this fact, we first write $x=z\log k$. We then show that the derivative with respect to $z$ of the left hand side of \eqref{i:1.2.1} is positive and also that it is positive at $z=1.8$ for each $k \ge 44$. Then, since $x \ge \log\log n_k > \log k$ \eqref{lem2-3}, we can now assume that $x \in [\log k,1.8\log k]$.

It remains to verify that $\rho(n) < 2$ for each $x \in [\log\log n_k,1.8\log k]$ and each fixed value of $k$ with a limited number of computations. To do so, we will work with the function
\begin{equation}\label{s2;r,ram}
h_k(x):=\Bigl(\frac{\log(\exp(x)+\log n_k)-\log k-\frac{1}{k}\sum_{j=1}^{k}\log\log p_j}{\log(1+\frac{\exp(x)}{kx})}-1\Bigr)\frac{x}{\log x},
\end{equation}
which is the right hand side of \eqref{s3,i1}, and use Lemma \ref{lem_6}.

Let us first establish that
\begin{equation}\label{maj:der}
\max_{x \in [\log\log n_k,1.8\log k]}|h'_k(x)| \le\frac{400}{81}\frac{\log^2 k}{\log\log k}+\frac{130}{27}\frac{\log k}{\log\log k}.
\end{equation}
We have
\begin{equation}\label{diff_1}
h'_k(x)=\Bigl(\frac{1}{\log x}-\frac{1}{\log^2 x}\Bigr)\Bigl(\frac{W}{\ell}-1\Bigr)+\frac{x}{\log x}\Bigl(\frac{\exp(x)}{\ell(\exp(x)+\log n_k)}-\frac{W}{\ell^2}\frac{\frac{1}{\vartheta}-\frac{1}{x\vartheta}}{1+\frac{1}{\vartheta}}\Bigr),
\end{equation}
where we wrote $\ell:=\log(1+\frac{1}{\vartheta})$ and $W:=\log(\exp(x)+\log n_k)-\log k-\frac{1}{k}\sum_{j=1}^{k}\log\log p_j$ to simplify. Clearly,
\begin{equation}\label{s3,i3}
|h'_k(x)| \le \frac{1}{\log x}\Bigl|\frac{W}{\ell}-1\Bigr|+\frac{x}{\log x}\Bigl(\frac{1}{\ell}+\frac{W}{\ell^2}\Bigr) \quad (x \ge \exp(1)).
\end{equation}
From there, we use the fact that $\vartheta \le 1.39$, i.e. $\ell > \frac{27}{50}$, and we will establish that $W < 0.8\log k$ uniformly for $x \in [\log\log n_k,1.8\log k]$. It allows us to conclude that \eqref{maj:der} holds since $-1 < \frac{W}{\ell}-1 < \frac{W}{\ell}$ given that $\log 2 \le \frac{\log \tau(n)}{k} \le W$ so that $\Bigl|\frac{W}{\ell}-1\Bigr| < \frac{50}{27}W$. Now, from \eqref{lem2-1} and \eqref{lem2-2},
\begin{eqnarray*}
W & < & \log (k^{1.8}+2k\log k)-\log k-\log\log k\\
& = & \log (\frac{k^{0.8}}{\log k}+2) \le 0.8\log k
\end{eqnarray*}
and the desired inequality follows.

We verify that the right hand side of \eqref{maj:der} is an increasing function of $k$ on the interval $[44,10999]$. Thus it is less than $M_1 = 215$. For this reason, we set $\delta = 0.002$ and we thus have $\delta M_1 = 0.43$. Using Maple, we evaluate the right side of $h_k(x)$ at each step of 0.002 in the interval $x \in [\log k,1.8\log k]$ for each $k \in [44,10999]$. This verification finishes the proof that $\rho(n) < 2$ for each such value of $k$.

\subsection{The case $1 \le k \le 43$}

In this section, we finishes the proof of Theorem \ref{thm_1}. We will see that only the case $k=2$ has some values of $n$ for which $\rho(n) \ge 2$. The general idea is the same as in the previous section. We now set up what is needed to use Lemma \ref{lem_6} on the function $h_k(x)$.

Obviously we have $x \ge \log\log n_k$ and our main objective is to find an upper bound for an $x$ that would realize $\rho(n) \ge 2$. In order to do that, we start from \eqref{s3,i1} and write
\begin{eqnarray} \nonumber
\rho(n) & \le & \Bigl(\frac{\log(\exp(x)+\log n_k)-\log k-\frac{1}{k}\sum_{j=1}^{k}\log\log p_j}{\log(1+\frac{1}{\vartheta})}-1\Bigr)\frac{x}{\log x}\\ \nonumber
& \le & \Bigl(\frac{\log(2\exp(x))-\log k-\frac{1}{k}\sum_{j=1}^{k}\log\log p_j}{x-\log x-\log k}-1\Bigr)\frac{x}{\log x}\\ \label{s4,i1}
& = & \frac{x\log 2x-\frac{x}{k}\sum_{j=1}^{k}\log\log p_j}{(x-\log x-\log k)\log x}.
\end{eqnarray}
We then show that \eqref{s4,i1} is strictly less than 2 for $x \ge 9.36$ for each $k \in [1,43]$. Now that we have the desired upper bound for $x$, we are ready for the final verification. From a previous verification, we know that  $x \ge \log\log 10^9$. We need an upper bound for $|h'_k(x)|$ with $x \in [\log\log \max(10^9,n_k),9.36]$ and for that we use \eqref{s3,i3}. We still have $\ell \ge \frac{27}{50}$ and we get an upper bound for $W$ directly from the fact that $x \le 9.36$. We find that
$$
\max_{x \in [\log\log \max(10^9,n_k),9.36]}h'_k(x) \le 165=:M_1.
$$
For this reason, we choose $\delta = 0.00004$ so that $\delta M_1=0.0066$. We verify with a computer at each step of $0.00004$ in $[\log\log \max(10^9,n_k),9.36]$ for $k=1$ and $k \in [3,43]$ and call the maximum $M$. We find that $M+\delta M_1 < 2$. Finally, for $k = 2$, we verify that $\max_{x \in [4,9.36]}h_2(x) < 2$ using the same method. For $x \le 4$, it is enough to verify the numbers of the shape $2^a\cdot 3^b$ with $a \le 77$ and $b \le 49$ that are larger than $10^9$. We find that the maximum is only realized by $n=2^{26} \cdot 3^{16}$. The proof is complete.

\section{Proof of Theorem \ref{thm_2}}

Let us fix $\vartheta \in (0,1]$ and choose an $\epsilon$ satisfying $\frac{\vartheta}{100} \ge\epsilon > 0$. We can assume that $\vartheta \in (1/(s+1),1/s]$ for some positive integer $s$. It is enough to prove the result for primary integers. We consider the ordered set $\mathcal{H}_{\epsilon}$ of primary integers for which
$$
\Bigl|\omega(n) - \frac{\vartheta \log n}{\log\log n}\Bigr| < \frac{\epsilon\log n}{\log\log n}.
$$
We define the constant
$$
\mu := \limsup_{\substack{n \rightarrow \infty \\ n \in \mathcal{H_{\epsilon}}}} \frac{\log \tau(n)\log\log n}{\log n}.
$$
Let
$$
\mathcal{H}_{\epsilon}^*:=\{n \in \mathcal{H}_{\epsilon} \cap \mathbb{R}_{>x_0}:\ \Bigl|\frac{\log \tau(n)\log\log n}{\log n}-\mu\Bigr| \le \epsilon^2\}
$$
where $x_0(=x_0(\epsilon))$ is chosen large enough so that $\frac{\log \tau(n)\log\log n}{\log n} \ge \mu + \epsilon^2$ is impossible in $\mathcal{H}_{\epsilon}$. Each primary integer $u$ can be written uniquely as $u:=u^{**}\cdot u^{*} \cdot u^{(s+1)}\cdots u^{(1)}$ where
$$
u^{(j)}:=\prod_{\substack{p \mid u \\ p^j \| u}}p^j \quad (j=1,\dots,s+1)
$$
and where $u^{**}$ is the divisor of $u$ formed of the (at most) $\lfloor\epsilon\omega(u)\rfloor$ first prime numbers. Let us assume, for a contradiction, that for some $n \in \mathcal{H}_{\epsilon}^{*}$ large enough we have $\omega(n^*) \ge \lfloor\epsilon\omega(n)\rfloor$. We will then find a primary integer $n'$ satisfying
$$
\omega(n') < \frac{(\vartheta+\epsilon)\log n'}{\log\log n'},
$$
$$
n \le n' \le n\exp\Bigl(\epsilon\log(c_1/\epsilon)\frac{\log n}{\log\log n}(1+o(1))\Bigr),
$$
for some constant $c_1$, and for which $\tau(n') \ge (1+\frac{1}{8s^2})^{\lfloor\epsilon\omega(n)\rfloor}\cdot \tau(n)$. So we will have
\begin{equation}
\exp\Bigl((\mu+c_2\epsilon)\frac{\log n}{\log\log n}\Bigr)  \le \tau(n') \le  \exp\Bigl((\mu+\epsilon^2)\frac{\log n'}{\log\log n'}\Bigr),
\end{equation}
where $c_2$ is a constant depending only on $\theta$, from which we find a contradiction for $\epsilon$ small enough when $n$ is large enough. This means that in fact $\omega(n^*) < \epsilon\omega(n)$ and we have established that, for $n$ large enough, $n:=n'' \cdot n^{(s+1)}\cdots n^{(1)}$ where $n''$ is made of at most the first $2\epsilon\omega(n)$ prime numbers.

We are thus ready to define this integer. We verify that the transformation of $n$ which consists in replacing the largest prime factor $q_1$ of $n^*$ by the smallest prime factor $q_2$ of $n^{(j)}$ (for the smallest $j \in \{1,\dots,s\}$ available), i.e. $n \mapsto \frac{q_2 n}{q_1}$, increases the value of $\tau$ by a factor $\ge 1+\frac{1}{8s^2}$, increases the integer by a factor $\ll 1/\epsilon$ and transform $n$ into a new primary integer. Since $\vartheta > \frac{1}{s+1}$, it is possible to iterate this transformation $\lfloor\epsilon\omega(n)\rfloor$ times for $\epsilon$ small enough and $n$ large enough. By doing so, starting with $n$, we end with an integer $n'$ satisfying the 3 announced properties.

Now, from Theorem \ref{thm_1} we have
$$
\tau(n'') \le \exp\Bigl(c_3 \epsilon\log(1/\epsilon)\frac{\log n}{\log\log n}\Bigr)
$$
for some constant $c_3$. We deduce that $\tau(n'')$ is small when compared to $\tau(n)$. We can thus consider the integer $m:=n/n''=n^{(1)}\cdots n^{(s+1)}$ and optimize the value of $\tau(m)$ under the condition
$$
\omega(m) = \frac{(\vartheta+O(\epsilon))\log n}{\log\log n}.
$$
By writing $k_1+\dots+k_{s+1}=\omega(m)=:k$, i.e. $k_j:=\omega(n^{(j)})$, we find
$$
\log(\tau(m)) = k_1\log(2)+\cdots+k_{s+1}\log(s+2)
$$
so that $\tau(m)$ is maximal when most of the $k_j$ with $j$ small are zero. For this reason, we will assume that $k_1,\dots,k_{s-1}=0$ and that only $k_s$ and $k_{s+1}$ may be nonzero. We write $k_s=\alpha k$, so that $k_{s+1}=(1-\alpha)\cdot k$. From there, we can assume that $m=n^{(s)}\cdot n^{(s+1)}$ and now the problem is to maximize
\begin{equation}\label{th-2,ineq}
\log(\tau(m)) = \bigl(\alpha\log(s+1)+(1-\alpha)\log(s+2)\bigr) \cdot k
\end{equation}
under $m \le n$ which can be written as
\begin{eqnarray*}
\log n & \ge & s\sum_{p \mid n_s}\log p+(s+1)\sum_{p \mid n_{s+1}}\log p \\
& = & \bigl(s\alpha k\log(\alpha k)+(s+1)(1-\alpha) k\log((1-\alpha) k)\bigr)(1+o(1)) \\
& = & \bigl(s\alpha+(s+1)(1-\alpha)\bigr)\cdot(k\log k)\cdot(1+o(1)) \\
& = & \bigl(s\alpha+(s+1)(1-\alpha)\bigr)\cdot(\vartheta+O(\epsilon))\cdot(\log n)\cdot(1+o(1)).
\end{eqnarray*}
We deduce that $\alpha \ge s+1-1/\vartheta+O(\epsilon)$ and by using this inequality in \eqref{th-2,ineq} we get
\begin{eqnarray*}
\log(\tau(m)) & \le & \bigl((\vartheta(s+1)-1)\cdot\log(s+1)+(1-s\vartheta)\cdot\log(s+2)+O(\epsilon)\bigr) \cdot \frac{\log n}{\log\log n}\\
& = & \bigl(f(\vartheta)+O(\epsilon)\bigr) \cdot \frac{\log n}{\log\log n}
\end{eqnarray*}
so that
$$
\log(\tau(n)) \le \bigl(f(\vartheta)+O(\epsilon\log(1/\epsilon))\bigr) \cdot \frac{\log n}{\log\log n}.
$$
This is the desired upper bound.  

For the lower bound, we choose a large $z$ and we construct an integer $m=m_1^{s+1}m_2^{s}$ such that
$$
m_1=\prod_{j \le (1-s\vartheta)\frac{\log z}{\log\log z}}p_j \quad \mbox{and} \quad m_2=\prod_{\substack{j \le \vartheta\frac{\log z}{\log\log z}\\ p\nmid m_1}}p_j.
$$
We verify that $m=z\exp\bigl(O\bigl(\frac{\log z \log\log\log z}{\log\log z}\bigr)\bigr)$. The proof is complete.

\section{Concluding remarks}

We have seen that $\rho(n) < 2$ given $\omega(n) \neq 2$. One can wonder if it is a good inequality. We can see directly from \eqref{s3,i2} that $\limsup_{\substack{n \rightarrow \infty \\ \omega(n)=k}}\rho(n) \le 1$. Let us show that we have in fact equality. Indeed, let $(z_i)_{i \ge 1}$ be a strictly increasing sequence of integers large enough. There are $\theta_1,\dots,\theta_k$ satisfying $\max_{j} |\theta_j| < 1/2$ such that
$$
\prod_{j=1}^{k}\Bigl(\frac{\log z_i n_k}{k\log p_j}+\theta_j\Bigr) =:\prod_{j=1}^{k}(\alpha_j+1)=\tau(m_i)
$$
which defines the integer $m_i:=p_1^{\alpha_1}\cdots p_k^{\alpha_k}$. We verify that
$$
\log \tau(m_i)=k\log(\exp(x_i)+\log n_k)-k\log k-\sum_{j=1}^{k}\log\log p_j +O\Bigl(\frac{k\log n_k}{\exp(x_i)}\Bigr)
$$
where $x_i:=\log\log m_i$. Thus, by using this value of $\log \tau(m_i)$ in \eqref{rho} (as we did to obtain \eqref{s3,i1}) we deduce as above that $\limsup_{\substack{n \rightarrow \infty \\ \omega(n)=k}}\rho(n) \ge 1$.

Let us now prove that
$$
\limsup_{n \rightarrow \infty}\rho(n) = 1.
$$
We already have the lower bound for each single value of $k \ge 1$. For the upper bound we use the main argument of Section \ref{t1.1}. Precisely, for each fixed $\epsilon > 0$ we have

$$
\frac{1}{(\vartheta+1)\log(1+\frac{1}{\vartheta})}\frac{\log x+t}{x-\log x-t} \le \frac{\log x}{x}
$$
for $t \le (1-\epsilon)x$ when $x$ is large enough. In the case where $t > (1-\epsilon)x$, we use inequality \eqref{s3,i2} to find $\limsup_{n \rightarrow \infty}\rho(n) \le \frac{1}{1-\epsilon}$.

{\it E-mail address:} {\tt Patrick.Letendre.1@ulaval.ca}

\end{document}